\documentclass[12pt]{amsart}\usepackage{latexsym,amsmath,amssymb,graphicx,color}
\usepackage{tikz}\usetikzlibrary{matrix,arrows,calc}
\graphicspath{{.}{./figs/}}

\theoremstyle{plain}\newtheorem{thm}{Theorem}[section]
\newtheorem{lem}[thm]{Lemma}\newtheorem{prop}[thm]{Proposition}

\theoremstyle{definition}\newtheorem{defn}[thm]{Definition}
\newtheorem{rem}[thm]{Remark}

\newcommand{\N}{\mathbb{N}}
\newcommand{\R}{\mathbb{R}}
\newcommand{\one}{\mathbf{1}}\newcommand{\zer}{\mathbf{0}}
\newcommand{\sph}{\mathbf{S}}
\newcommand{\fix}{\textsc{Fix}}\newcommand{\ms}{\textsc{Min}}
\newcommand{\dir}{\textsc{Dir}}\newcommand{\mov}{\textsc{Mov}}
\newcommand{\lin}{\textsc{Lin}}\newcommand{\aff}{\textsc{Aff}}
\newcommand{\spn}{\textsc{Span}}\newcommand{\isom}{\textsc{Isom}}
\newcommand{\cay}{\textsc{Cay}}\newcommand{\DM}{\textsc{DM}}
\newcommand{\inv}{{inv}}
\newcommand{\onto}{\twoheadrightarrow}
\newcommand{\bigjoin}{\bigvee}\newcommand{\bigmeet}{\bigwedge}
\newcommand{\join}{\vee}\newcommand{\meet}{\wedge}
\newcommand{\lr}{\ell_R}\newcommand{\ls}{\ell_S}

\title[Euclidean Isometries]{Factoring euclidean isometries}\date{\today}

\author{Noel Brady}\address{Dept. of Math. U. of Oklahoma, Norman, OK 73019}
\email{nbrady@math.ou.edu}

\author{Jon McCammond}\address{Dept. of Math. UC Santa Barbara, Santa
  Barbara, CA 93106}
\email{jon.mccammond@math.ucsb.edu}

\begin{document}

\begin{abstract}
  Every isometry of a finite dimensional euclidean space is a product
  of reflections and the minimum length of a reflection factorization
  defines a metric on its full isometry group.  In this article we
  identify the structure of intervals in this metric space by
  constructing, for each isometry, an explicit combinatorial model
  encoding all of its minimal length reflection factorizations.  The
  model is largely independent of the isometry chosen in that it only
  depends on whether or not some point is fixed and the dimension of
  the space of directions that points are moved.
\end{abstract}


\maketitle

Every good geometry book proves that each isometry of euclidean
$n$-space is a product of at most $n+1$ reflections and several
more-advanced sources include Scherk's theorem which identifies the
minimal length of such a reflection factorization from the basic
geometric attributes of the isometry under consideration
\cite{Scherk50, Dieudonne71, SnapperTroyer89, Taylor92}.  The
structure of the full set of minimal length reflection factorizations,
on the other hand, does not appear to have been given an elementary
treatment in the literature even though the proof only requires basic
geometric tools.\footnote{The only discussions of this issue that we
  have found in the literature use Wall's parameterization of the
  orthogonal group and the main results are stated in terms of
  inclusions of nondegenerate subspaces under an asymmetric bilinear
  form derived from the original isometry \cite{Taylor92, Wall63}.}
In this article we construct, for each isometry, an explicit
combinatorial model encoding all of its minimal length reflection
factorizations.  The model is largely independent of the isometry
chosen in that it only depends on whether or not some point is fixed
and the dimension of the space of directions that points are moved.

Analogous results for spherical isometries already exist and are easy
to state: when $w$ is an orthogonal linear transformation of $\R^n$
only fixing the origin, for example, there is a natural bijection
between minimal length factorizations of $w$ into reflections fixing
the origin and complete flags of linear subspaces in $\R^n$
\cite{BradyWatt02}.  In other words, the structure of all such
factorizations is encoded in the lattice of linear subspaces of $\R^n$
with one factorization for each maximal chain.  We construct a similar
poset for euclidean isometries but its structure is more complicated.
The motivation for constructing these combinatorial models is to use
them to analyze the structure of euclidean Artin groups.  See
\cite{Mc-lattice} and \cite{McSu-euclidean} for details.  It is with
this application in mind that we include an examination of the extent
to which these posets are lattices.

The article is structured as follows.  The first four sections
establish basic definitions, the middle sections define the
combinatorial models and establish our main result, and the final
sections explore the extent to which the constructed posets fail to be
complete lattices.

\section{Euclidean geometry}\label{sec:euclid}

In this section we review some elementary euclidean geometry with a
special emphasis on notation.  The main thing to note is that we
sharply distinguish between points and vectors as in
\cite{SnapperTroyer89} since this greatly clarifies the arguments used
latter in the article.

\begin{defn}[Points and vectors]\label{def:pt-vec}
  Throughout the article, $V$ denotes an $n$-dimensional real vector
  space with a positive definite inner product and $E$ denotes its
  affine analog where the location of the origin has been forgotten.
  The elements of $V$ are \emph{vectors} and the elements of $E$ are
  \emph{points}.  We use greek letters for vectors and roman letters
  for points.  There is a uniquely transitive action of $V$ on $E$.
  Thus, given a point $x$ and a vector $\lambda$ there is a unique
  point $y$ with $x + \lambda = y$ and given two points $x$ and $y$
  there is a unique vector $\lambda$ with $x + \lambda = y$.  We say
  that $\lambda$ is the vector from $x$ to $y$.  For any $\lambda \in
  V$, the map $x\mapsto x+\lambda$ is a \emph{translation} isometry
  $t_\lambda$ of $E$ and since $t_\mu t_\nu = t_{\mu + \nu} = t_\nu
  t_\mu$, the set $T = \{ t_\lambda \mid \lambda \in V\}$ is an
  abelian group.  For any point $x \in E$, the map $\lambda \mapsto
  x+\lambda$ is a bijection that identifies $V$ and $E$ but the
  isomorphism depends on this initial choice of a \emph{basepoint} $x$
  in $E$.  Lengths of vectors and angles between vectors are
  calculated using the usual formulas and distances and angles in $E$
  are defined by converting to vector-based calculations.
\end{defn}

\begin{defn}[Linear subspaces of $V$]
  A \emph{linear subspace} of $V$ is a subset closed under linear
  combination and every subset of $V$ is contained in a unique minimal
  linear subspace called its \emph{span}.  Each linear subspace $U$
  has an \emph{orthogonal complement} $U^\perp$ consisting of those
  vectors in $V$ orthogonal to all the vectors in $U$ and there is a
  corresponding orthogonal decomposition $V = U \oplus U^\perp$.  The
  \emph{codimension} of $U$ is the dimension of $U^\perp$.
\end{defn}

\begin{defn}[Affine subspaces of $E$]
  An \emph{affine subspace} of $E$ is any subset $B$ that contains
  every line determined by distinct points in $B$ and every subset of
  $E$ is contained in a unique minimal affine subspace called its
  \emph{affine hull}.  Associated with any affine subspace $B$ is its
  (linear) \emph{space of directions} $\dir(B) \subset V$ consisting
  of the collection of vectors connecting points in $B$.  The
  \emph{dimension} and \emph{codimension} of $B$ is that of its space
  of directions and a set of $k+1$ points in $E$ is in \emph{general
    position} when its affine hull has dimension $k$.
\end{defn}

\begin{defn}[Barycentric coordinates]
  Let $x_0, x_1, \ldots, x_k$ be $k+1$ points in general position in
  $E$, and let $B$ be its $k$-dimensional affine hull.  If we identify
  $E$ and $V$ by picking a basepoint than each point in $B$ can be
  expressed as a linear combination of the vectors from this origin to
  $x_i$ where the coefficients sum to $1$.  Since these coefficients
  turn out to be independent of our choice of basepoint, we can
  unambiguously write $\sum c_ix_i = c_0 x_0 + c_1 x_1 +\cdots +
  c_kx_k$ (with $c_i \in \R$ and $\sum c_i = 1$) to represent each
  point in $B$ without making such an identification.  These are
  called \emph{barycentric coordinates} on $B$.
\end{defn}

The intrinsic way barycentric coordinates are defined means when
$\{x_i\}$ and $\{y_i\}$ are points in general position with affine
hulls $B$ and $C$ and an isometry of $E$ sends each $x_i$ to $y_i$,
then it also sends the point $\sum c_i x_i \in B$ to the point $\sum
c_i y_i \in C$ with the same coefficients.

\begin{defn}[Affine subspaces of $V$]
  An \emph{affine subspace of $V$} is a translation of one of its
  linear subspaces.  In particular, every affine subspace $M$ can be
  written in the form $t_\mu(U) = U +\mu = \{ \lambda + \mu \mid
  \lambda \in U\}$ where $U$ is a linear subspace of $V$.  This
  representation is not unique, (since $U + \mu = U$ for all $\mu \in
  U$), but it can be made unique if we insist that $\mu$ is of minimal
  length or, equivalently, that $\mu$ be a vector in $U^\perp$, in
  which case we say $U +\mu$ is the \emph{standard form} of $M$.
\end{defn}

Under any identification of $E$ and $V$, the affine subspaces of $E$
are identified with those of $V$.  More precisely, each affine
subspace $B$ in $E$ corresponds to some $M = U + \mu$.  The linear
subspace $U = \dir(B)$ is canonical but the vector $\mu\in U^\perp$
depends on the choice of basepoint.

\section{Posets and lattices}\label{sec:poset}

For posets and lattices we generally follow \cite{EC1} and
\cite{DaPr02}.

\begin{defn}[Posets]\label{def:poset}
  Let $P$ be a partially ordered set.  When a minimum or a maximum
  element exists in $P$, it is denoted $\zer$ and $\one$,
  respectively, and posets containing both are \emph{bounded}. The
  \emph{dual} $P^*$ of a poset $P$ has the same underlying set but
  with the order reversed, and a poset is \emph{self-dual} when it and
  its dual are isomorphic.  For each $Q\subset P$ there is an induced
  \emph{subposet} structure on $Q$ which is simply the restriction of
  the order on $P$.  A subposet $C$ in which any two elements are
  comparable is called a \emph{chain} and its length is $\vert C \vert
  -1$.  Every finite chain is bounded and its maximum and minimum
  elements are its \emph{endpoints}.  If a finite chain $C$ is not a
  subposet of a strictly larger finite chain with the same endpoints,
  then $C$ is \emph{saturated}.  Saturated chains of length~$1$ are
  called \emph{covering relations}.  If every saturated chain in $P$
  between the same pair of endpoints has the same finite length, then
  $P$ is \emph{graded}.  The \emph{rank} of an element $p$ is the
  length of the longest chain with $p$ as its upper endpoint and its
  \emph{corank} is the length of the longest chain with $p$ as its
  lower endpoint, assuming such chains exists.
\end{defn}

\begin{defn}[Lattices]\label{def:lattice}
  Let $Q$ be a subset of a poset $P$.  A lower bound for $Q$ is any
  $p\in P$ with $p \leq q$ for all $q\in Q$.  When the set of lower
  bounds for $Q$ has a unique maximum element, this element is the
  \emph{greatest lower bound} or \emph{meet} of $Q$.  Upper bounds and
  the \emph{least upper bound} or \emph{join} of $Q$ are defined
  analogously.  The meet and join of $Q$ are denoted $\bigmeet Q$ and
  $\bigjoin Q$ in general and $u \meet v$ and $u \join v$ if $u$ and
  $v$ are the only elements in $Q$.  When every pair of elements has a
  meet and a join, $P$ is a \emph{lattice} and when every subset has a
  meet and a join, it is a \emph{complete lattice}.  An easy argument
  shows that every bounded graded lattice is complete.
\end{defn}

The main posets we need are the linear and affine subspace posets.

\begin{defn}[Linear subspaces]\label{def:lin}
  The linear subspaces of $V$ partially ordered by inclusion form a
  poset we call $\lin(V)$.  It is graded, bounded, self-dual and a
  complete lattice.  The bounding elements are clear, the grading is
  by dimension (in that a $k$-dimensional subspace has rank~$k$ and
  corank~$n-k$), the meet of a collection of subspaces is their
  intersection and their join is the span of their union.  And
  finally, the map sending a linear subspace to its orthogonal
  complement is a bijection that establishes self-duality.
\end{defn}

\begin{defn}[Affine subspaces]\label{def:aff}
  The \emph{affine subspaces} of $E$ partially ordered by inclusion
  form a poset we call $\aff(E)$.  For each $n$, it is a graded poset
  that is bounded above but not below since distinct points are
  distinct minimal elements.  It is neither self-dual nor a lattice.
  There is, however, a well-defined rank-preserving poset map $\aff(E)
  \onto \lin(V)$ sending each affine subspace $B$ to its space of
  directions $\dir(B)$.  Chains in $\lin(V)$ and $\aff(E)$ are
  traditionally called \emph{flags} and a maximal chain, starting at a
  minimal element and ending at the whole space, is a \emph{complete
    flag}.
\end{defn}

\section{Euclidean isometries}

The symmetries of $E$ that preserve distances and angles are
isometries and they form a group that we call $W = \isom(E)$.  The
choice of the letter $W$ reflects the fact that we treat $\isom(E)$
like a continuous version of an affine Weyl group.  We begin by
highlighting two sets associated with any isometry.

\begin{defn}[Invariants]\label{def:inv}
  Let $w\in W$ be an isometry of $E$.  If $\lambda$ is the vector from
  $x$ to $w(x)$ then we say $x$ is \emph{moved by $\lambda$ under
    $w$}.  The collection $\mov(w) = \{ \lambda \mid x+\lambda = w(x),
  x \in E\} \subset V$ of all such vectors is the \emph{move-set} of
  $w$.  As we show below, a move-set is an affine subspace and thus
  has standard form $U +\mu$ where $U$ is a linear subspace and $\mu$
  is a vector in $U^\perp$.  The points in $E$ that are moved by $\mu$
  under $w$ are those that are moved the shortest distance.  The
  collection $\ms(w)$ of all such points is called the \emph{min-set}
  of $w$.  The sets $\mov(w) \subset V$ and $\ms(w) \subset E$ are the
  \emph{basic invariants} of $w$.  The basic invariants of a glide
  reflection of the plane are illustrated in Figure~\ref{fig:glide}.
\end{defn}

\begin{figure}
\begin{tabular}{ccc}
\begin{tabular}{c}
\includegraphics[scale=.85]{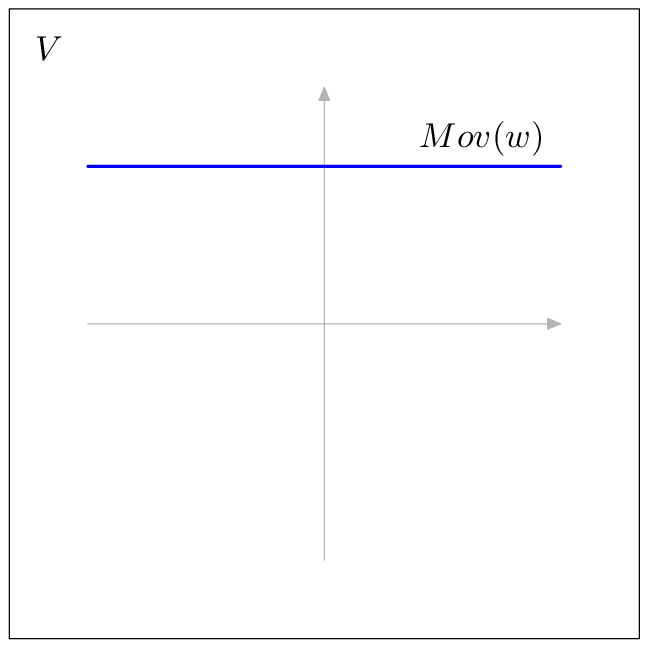}
\end{tabular}
& \hspace*{0em} &
\begin{tabular}{c}
\includegraphics[scale=.85]{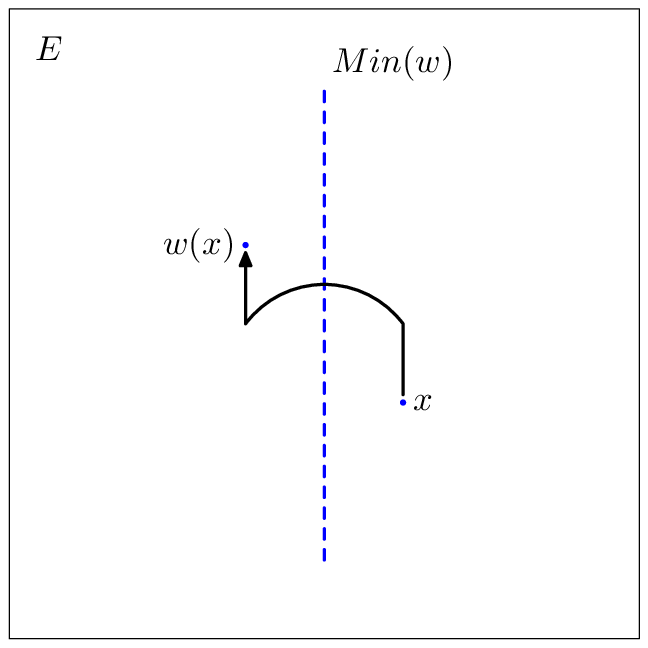}
\end{tabular}
\end{tabular}
\caption{The basic invariants of a glide reflection.\label{fig:glide}}
\end{figure}

\begin{prop}[Invariants are affine]\label{prop:inv-aff}
  For each isometry $w\in W$, its move-set $\mov(w)$ is an affine
  subspace of $V$ and for each affine subspace $M \subset \mov(w)$,
  the points moved by some $\lambda \in M$ under $w$ form an affine
  subspace of $E$.  In particular, the min-set $\ms(w)$ is an affine
  subspace of $E$.
\end{prop}

\begin{proof}
  For each $\mu, \nu \in \mov(w)$ choose $x$ and $y$ so that $w(x) = x
  + \mu$ and $w(y) = y + \nu$.  Because $w$ is an isometry, it sends
  the line through $x$ and $y$ to the line through $w(x)$ and $w(y)$.
  In barycentric coordinates, for each $c,d \in \R$ with $c+d=1$, the
  point $cx+dy$ is sent to $c(x+\mu) + d(y+\nu) =
  (cx+dy)+(c\mu+d\nu)$.  In particular, every point on the affine line
  through $\mu$ and $\nu$ in $V$ is the motion of some point in $E$
  showing that $\mov(w)$ is affine.  Similarly, if $x$ and $y$ are
  both moved by $\lambda$ and $\mu$ in an affine subspace $M$ under
  $w$, then for each $c,d \in \R$ with $c+d=1$, $cx+dy$ is sent to
  $c(x+\lambda) + d(y+\mu) = (cx+dy)+(c\lambda+d\mu)$.  Thus every
  point on the line through $x$ and $y$ is moved by some vector in $M$
  under $w$ and the set of all such points is an affine subspace.
\end{proof}

\begin{defn}[Elliptic and hyperbolic]
  Let $w$ be an isometry of $E$ and let $U+\mu$ be the standard form
  of its move-set $\mov(w)$.  There are points fixed by $w$ iff $\mu$
  is trivial iff $\mov(w)$ is a linear subspace.  Under these
  conditions $w$ is \emph{elliptic} and the min-set $\ms(w)$ is just
  the \emph{fix-set} $\fix(w)$ of points fixed by $w$.  Similarly, $w$
  has no fixed points iff $\mu$ is nontrivial iff $\mov(w)$ a
  nonlinear affine subspace of $V$.  Under these conditions $w$ is
  \emph{hyperbolic}. The names come from a tripartite classification
  of isometries \cite{BridsonHaefliger99}. The third type, parabolic,
  does not occur in this context.
\end{defn}

Translations and reflections are the simplest examples of hyperbolic
and elliptic isometries, respectively.  The translations $t_\lambda$,
defined in Definition~\ref{def:pt-vec}, can alternatively be
characterized as the only isometries whose move-set is a single point
or whose min-set is all of $E$.  They are the essential difference
between elliptic and hyperbolic isometries.  There is a special
factorization of a hyperbolic isometry $w$ that we call its
\emph{standard splitting}.

\begin{prop}[Standard splittings]\label{prop:std-split}
  If $w$ is a hyperbolic isometry whose move-set has standard form
  $\mov(w) = U+\mu$, then the unique isometry $u$ that satisfies the
  equation $w = t_\mu u$, is an elliptic isometry with $\mov(u) =
  \dir(\mov(w))$ and $\fix(u) = \ms(w)$.
\end{prop}

\begin{proof}
  First note that $\dir(\mov(w)) = U$ and translations translate
  move-sets.  Next, a point $x$ is fixed by $u$ iff $x$ is moved by
  $\mu$ under $w$ and such points do exist.  Thus $u$ is elliptic and
  $\fix(u) = \ms(w)$.
\end{proof}

The standard splitting can be used to show that min-sets and move-sets
have complementary dimensions.

\begin{lem}[Complementary invariants]\label{lem:inv-comp}
  For each isometry $w\in W$ the dimensions of $\mov(w)$ and $\ms(w)$
  add up to the dimension of $E$.
\end{lem}

\begin{proof}
  Suppose $w$ is elliptic and $U = \mov(w)$. If we choosing a
  basepoint in $\fix(w)$ to identify $E$ with $V$, then $w$
  corresponds to a linear transformation of $V$ and the map $x\mapsto
  w(x)-x$ is linear as well. Its image is $\mov(w)$, its kernel is
  $\fix(w)$ and the desired equation is just the rank-nullity theorem.
  When $w$ is hyperbolic, we use Propositinon~\ref{prop:std-split} to
  find an elliptic isometry $u$ with $\mov(u) = \dir(\mov(w))$ and
  $\fix(u) = \ms(w))$.  The result holds for $u$ and thus for $w$.
\end{proof}

There is also a stronger version of Lemma~\ref{lem:inv-comp}.

\begin{lem}[Orthogonal invariants]\label{lem:inv-orth}
  For each isometry $w \in W$, there is an orthogonal decomposition
  $V=\dir(\mov(w)) \oplus \dir(\ms(w))$.
\end{lem}

\begin{proof}
  By Lemma~\ref{lem:inv-comp} the dimensions of these subspaces add up
  to $n$ so it is sufficient to prove that $\dir(\mov(w))$ is subset
  of $\dir(\ms(w))^\perp$.  When $w$ is elliptic, $\dir(\mov(w)) =
  \mov(w)$ and $\dir(\ms(w)) = \dir(\fix(w))$ are linear subspaces of
  $V$.  To see that they are orthogonal, let $B = \fix(w)$ and let $U
  = \dir(B)^\perp$.  For each $x \in E$, there is a unique $x_0 \in B$
  closest to $x$ and the vector $\mu$ from $x$ to $x_0$ is orthogonal
  to $\dir(B)$, i.e. $\mu$ is in $U$.  Also, since $w$ preserves
  distances and fixes $x_0$, the point in $B$ closest to $w(x)$ is
  again $x_0$.  Thus the vector from $x$ to $w(x)$ is a combination of
  two vectors in $U$ and $\mov(w) \subset U$.  By the dimension count
  $\mov(w) = U$.  When $w$ is hyperbolic with $\mov(w) = U+\mu$, there
  is an elliptic isometry $u$ with $\mov(u) = U = \dir(\mov(w))$ and
  $\ms(u) = \ms(w)$ (Proposition~\ref{prop:std-split}).  The
  decomposition derived from $u$ completes the proof.
\end{proof}

The standard splitting also can be used to prove a handy
characterization of the min-set of an isometry.

\begin{prop}[Identifying min-sets]\label{prop:identify-min}
  If $w$ is a hyperbolic isometry whose move-set has standard form
  $\mov(w) =U + \mu$, then the min-set of $w$ is the unique affine
  subspace $B$ of $E$ stabilized by $w$ whose dimension is the
  codimension of $U$ and where all points in $B$ experience the same
  motion under $w$.
\end{prop}

\begin{proof}
  The min-set of $w$ has these properties since its points under the
  same motion by definition and the other aspects follow from
  Lemma~\ref{lem:inv-orth}.  Thus we only need to show that $\ms(w)$
  is the only such subspace.  Let $B$ be an affine subspace satisfying
  these conditions and choose an affine subspace $C$ in $E$ with
  $\dir(C) = U$.  The points in $C$ parameterize the affine subspaces
  $B'$ with $\dir(B') = U^\perp$ in the following sense: every such
  $B'$ intersects $C$ in a single point $x$ and each point $x \in C$
  determines a unique such subspace $B'$.  Call these the $U^\perp$
  subspaces of $E$.  Next, let $w= t_\mu u$ be the standard splitting
  of $w$ and note that because $\mov(u) = U$, $u$ stabilizes $C$.
  Moreover, because isometries send rectangles to rectangles, points
  in the same $U^\perp$ subspace undergo the same motion under the
  action of $u$ and thus the same motion under $w$.  Thus $C$ contains
  a representative of every possible motion under $u$ and because the
  dimension of $U$ and $C$ agree, it contains a unique representative
  of each motion.  In particular, $C$ contains a unique point that is
  fixed under $u$.  This also means that the converse also holds:
  points undergoing the same motion under $u$ (or equivalently the
  same motion under $w$) belong to the same $U^\perp$ subspace.  As a
  consequence $B$ is a subspace of a $U^\perp$ subspace and because
  their dimensions agree, $B$ is a $U^\perp$ subspace.  Finally, $\mu
  \in U^\perp$ means that $t_\mu$ stabilizes each $U^\perp$ subspace,
  and thus $w$ stabilizes such a subspace iff $u$ stabilizes it.  But
  the only $U^\perp$ subspace stabilized by $u$ is the one
  corresponding to the unique point it fixes in $C$.  These are
  exactly the points moved by $\mu$ under $w$ proving that $B$ is
  $\ms(w)$.
\end{proof}

\begin{defn}[Reflections in $V$]
  A \emph{hyperplane} $H$ in $V$ is a linear subspace of
  codimension~$1$ and for each hyperplane there is a unique nontrivial
  isometry $r$ that fixes $H$ pointwise called a \emph{reflection}.
  Let $L$ be the $1$-dimensional orthogonal complement of $H$ and note
  that it contains exactly two unit vectors $\pm \alpha$ called the
  \emph{roots} of $r$.  Since $L$, $H$, and $r$ can be recovered from
  $\alpha$, we write $L = L_\alpha$, $H = H_\alpha$ and $r =
  r_\alpha$.
\end{defn}

\begin{defn}[Reflections in $E$]\label{def:reflections}
  A \emph{hyperplane} $H$ in $E$ is an affine subspace of
  codimension~$1$ and, as above, for each hyperplane there is a unique
  nontrivial isometry $r$ that fixes $H$ pointwise called a
  \emph{reflection}.  The set $R$ of all reflections generates $W =
  \isom(E)$.  The space of directions is a hyperplane $\dir(H) =
  H_\alpha$ in $V$ and $\pm \alpha$ are called the \emph{roots} of
  $r$.  Note that $\mov(r)$ is the line $L_\alpha \subset V$ spanned
  by $\alpha$ and $\ms(r) = \fix(r) = H \subset E$.
\end{defn}

\section{Intervals in marked groups}\label{sec:intervals}

A \emph{marked group} is a group $G$ with a fixed generating set $S$
which, for convenience, we assume is symmetric and injects into $G$.
Thus, $s\in S$ iff $s^{-1} \in S$ and we can view $S$ as a subset of
$G$.  Fixing a generating set defines a natural metric on the group.

\begin{defn}[Metrics on groups]
  Let $G$ be a group generated by a set $S$.  The (right) \emph{Cayley
    graph of $G$ with respect to $S$} is a labeled directed graph
  denoted $\cay(G,S)$ with vertices indexed by $G$ and edges indexed
  by $G \times S$.  The edge $e_{(g,s)}$ has \emph{label} $s$, it
  starts at $v_g$ and ends at $v_{g'}$ where $g' = g\cdot s$. There is
  a natural faithful, vertex-transitive, label and orientation
  preserving left action of $G$ on its Cayley graph.  Moreover these
  are the only label and orientation preserving graph automorphisms of
  $\cay(G,S)$, making the identity automorphism the unique
  automorphism of this type that fixes a vertex.  The \emph{distance}
  $d(g,h)$ is the combinatorial length of the shortest path in the
  Cayley graph from $v_g$ to $v_h$.  Note that the symmetry assumption
  allows us to restrict attention to directed paths.  This defines a
  metric on $G$ and distance from the identity defines a \emph{length
    function} $\ls\colon G\to \N$.  The value $\ls(g) = d(1,g)$ is
  called the \emph{$S$-length of $g$} and it is also the length of the
  shortest factorization of $g$ in terms of elements of $S$.  Because
  Cayley graphs are homogeneous, metric properties of the distance
  function translate into properties of $\ls$.  Symmetry and the
  triangle inequality, for example, imply that $\ls(g) = \ls(g^{-1})$,
  and $\ls(gh)\leq \ls(g) + \ls(h)$.
\end{defn}

Next recall the notion of an interval in a metric space.

\begin{defn}[Intervals in metric spaces]
  Let $x$, $y$ and $z$ be points in a metric space $(X,d)$.  Borrowing
  from euclidean plane geometry we say that $z$ is \emph{between} $x$
  and $y$ whenever the triangle inequality degenerates into an
  equality.  Concretely $z$ is between $x$ and $y$ when $d(x,z) +
  d(z,y) = d(x,y)$.  The \emph{interval} $[x,y]$ is the collection of
  points between $x$ and $y$ and this includes both $x$ and $y$.
  Intervals can also be endowed with a partial ordering by declaring
  that $u \leq v$ whenever $d(x,u) + d(u,v) + d(v,y) = d(x,y)$.
\end{defn}

\begin{figure}
  \begin{tabular}{c}\includegraphics[scale=.4]{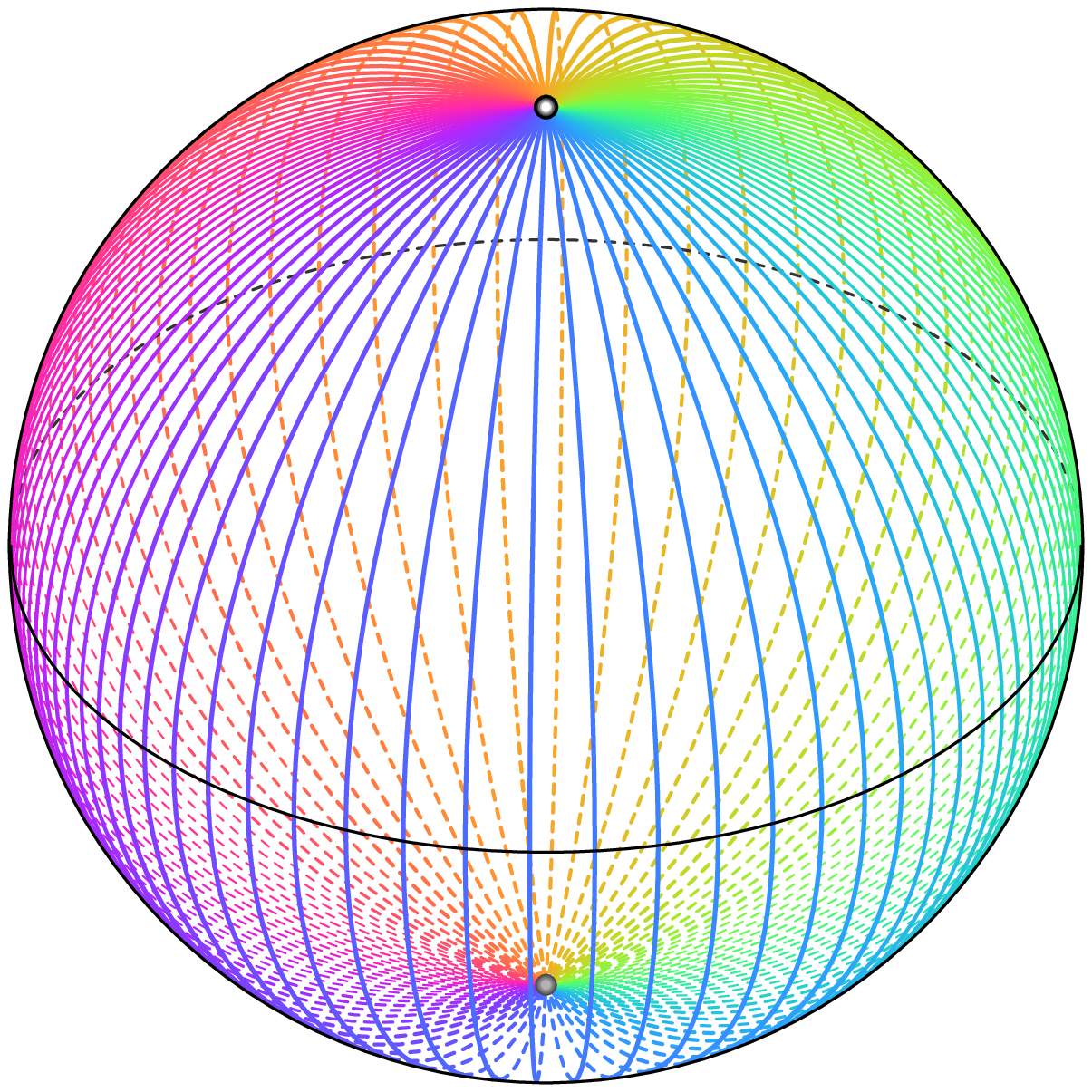}\end{tabular}
  \caption{The interval between two antipodal points of
    $\sph^2$.\label{fig:intervals}}
\end{figure}

As an illustration, consider points $x$ and $y$ on the $2$-sphere with
its usual metric.  If they are not antipodal, then the only points
between them are those on the unique shortest geodesic connecting $x$
to $y$ with the usual ordering along paths.  But if they are
antipodal, say $x$ is the south pole and $y$ is the north pole, then
the interval $[x,y]$ is all of $\sph^2$ and $u < v$ iff $u$ and $v$
lie on a common line of longitude connecting $x$ to $y$ with the
latitude of $u$ below the latitude of $v$ as shown in
Figure~\ref{fig:intervals}.  Because marked groups are metric spaces,
they have intervals.

\begin{defn}[Intervals in groups]
  Let $g$ and $h$ be distinct elements in a marked group $G$.  The
  \emph{interval} $[g,h]$ is the poset of group elements between $g$
  and $h$ with $g' \in G$ is in $[g,h]$ when $d(g,g') + d(g',h) =
  d(g,h)$ and $g' \leq g''$ when $d(g,g') + d(g',g'') + d(g'',h) =
  d(g,h)$.  In the Cayley graph $g'\in [g,h]$ means that $v_{g'}$ lies
  on some minimal length path from $v_g$ to $v_h$ and $g' < g''$ means
  that $v_{g'}$ and $v_{g''}$ both occur on a common minimal length
  path $v_g$ to $v_h$ with $v_{g'}$ occurring before $v_{g''}$.
\end{defn}

\begin{prop}[Posets in Cayley graphs]
  If $g$ and $h$ are distinct elements in a group $G$ generated by a
  set $S$ then the interval $[g,h]$ is a bounded graded poset whose
  Hasse diagram is embedded as a subgraph of the Cayley graph
  $\cay(G,S)$.
\end{prop}

\begin{proof}
  The interval $[g,h]$ is bounded below by $g$, bounded above by $h$
  and graded by the distance from $g$.  To see the Hasse diagram of
  $[g,h]$ inside the Cayley graph of $G$ note that its vertices
  correspond to the elements between $g$ and $h$ and its coverings
  relations correspond to those directed edges in the Cayley graph
  that occur in some shortest directed path from $v_g$ to $v_h$.
\end{proof}

Since the structure of a graded poset can be recovered from its Hasse
diagram, we let $[g,h]$ denote the edge-labeled directed graph that is
visible as a subgraph of the Cayley graph $\cay(G,S)$.

\begin{rem}[Isomorphic intervals]\label{rem:iso-int}
  The left action of a group on its right Cayley graph preserves
  labels and distances.  Thus the interval $[g,h]$ is isomorphic (as a
  edge-labeled directed graph) to the interval $[1,g^{-1}h]$.  In
  other words, every interval in the Cayley graph of $G$ is isomorphic
  to one that starts at the identity.
\end{rem}

We call $g^{-1}h$ the \emph{type} of the interval $[g,h]$ and note
that intervals are isomorphic iff they have the same type.  The
distance ordering on $G$ creates a single poset that contains every
type of interval.

\begin{defn}[Distance ordering]
  The \emph{distance ordering} on a marked group $G$ is defined by
  setting $g' \leq g$ iff $g' \in [1,g]$.  By
  Remark~\ref{rem:iso-int}, this gives $G$ a poset structure that
  contains an interval of every type that occurs in the metric space
  on $G$.
\end{defn}

\section{Reflection length}

In the language of the previous section, our goal is to establish the
poset structure of intervals in the group $W = \isom(E)$ generated by
the set $R$ of reflections.  The key to analyzing these intervals is
to have a good understanding of the length function.  In this section
we recall how the reflection length of an isometry is determined from
its basic invariants (Theorem~\ref{thm:lr}), a result known as
Scherk's theorem \cite{SnapperTroyer89}.  The lower bounds are
straight-forward.

\begin{prop}[Lower bounds]\label{prop:lower}
  If $w = r_1 r_2 \cdots r_k$ is a product of reflections, then the
  dimension of $\mov(w)$ is at most $k$ and it is equal to $k$ iff the
  roots of these reflections are linearly independent.  As a
  consequence $\lr(w) \geq \dim(\mov(w))$.  In addition, because
  linear independence of the roots implies $w$ is elliptic, the
  stronger lower bound $\lr(w) > \dim(\mov(w))$ holds when the
  isometry $w$ is hyperbolic.
\end{prop}

\begin{proof}
  Let $\pm \alpha_i$ be the roots of the reflection $r_i$.  Because
  $r_i$ only moves points in the $\alpha_i$ direction, the cumulative
  motion of any point $x$ under $w$ is a linear combination of the
  $\alpha_i$'s.  Thus $\mov(w)$ is contained in their span, proving
  the first assertion and its consequence.  The final part follows
  from the fact that hyperplanes with linearly independent normal
  vectors have a common point of intersection.  Such a point is not
  moved by any of the $r_i$ and thus is fixed by $w$.
\end{proof}

The easy way to establish an upper bound on reflection length is to
construct a factorization.  For this we need a few lemmas.

\begin{lem}[Fix-sets and reflections]\label{lem:fix-refl}
  If $w$ is an elliptic isometry and $r$ is a reflection whose
  hyperplane intersects $\fix(w)$, then $w'=rw$ is elliptic and the
  dimensions of $\fix(w)$ and $\fix(w')$ are at most $1$ apart.
\end{lem}

\begin{proof}
  Let $H$ be the hyperplane of $r$ and note that the hypothesized
  point in $H \cap \fix(w)$ shows that $w'=rw$ is elliptic and that
  $\fix(w)$ and $\fix(w')$ have points in common.  Moreover,
  $\fix(w')$ certainly contains $\fix(r) \cap \fix(w) = H \cap
  \fix(w)$, a space that is either $\fix(w)$ or a codimension~$1$
  subspace of $\fix(w)$. Thus the dimension of the fix-set decreases
  by at most $1$.  Finally, since reflecting twice is trivial, $rw' =
  w$ and $w'$ and $r$ satisfy the same hypotheses as $w$ and $r$.
  Reversing the roles of $w$ and $w'$ shows the dimension increases by
  at most $1$.
\end{proof}

\begin{lem}[Fixing points]\label{lem:fix-point}
  Let $w$ be a nontrivial elliptic isometry of $E$ whose fix-set is a
  $k$-dimensional affine subspace $B$. For each $(k+1)$-dimensional
  affine subspace $C$ containing $B$, there is a unique reflection $r$
  such that $w' = rw$ and $\fix(w') = C$.
\end{lem}

\begin{proof}
  Any point $x$ in $C \setminus B$ is not fixed by $w$ and the set of
  points equidistant from $x$ and $w(x)$ is a hyperplane $H$.  The
  reflection $r$ that fixes $H$ is the unique reflection sending
  $w(x)$ to $x$ and thus the unique reflection where $w' = rw$ fixes
  $x \in C$.  In other words this is the only reflection for which the
  assertion might be true.  Next, since $w$ is an isometry $d(x,y) =
  d(w(x),w(y)) = d(w(x),y)$ for all $y\in \fix(w)$.  Thus $B \subset
  H$ and all of $B$ is fixed by $w' = rw$.  In particular, $\fix(w')$
  contains $x$ and $B$ and by Proposition~\ref{prop:inv-aff} it
  contains their affine hull which is $C$.  By
  Lemma~\ref{lem:fix-refl} $\fix(w')$ cannot be an affine subspace
  properly containing $C$.
\end{proof}

These lemmas make it easy to construct reflection factorizations.

\begin{prop}[Elliptic upper bound]\label{prop:ell-upper}
  Every elliptic isometry $w$ with a $k$-dimensional move-set has a
  length~$k$ reflection factorization.  
\end{prop}

\begin{proof}
  By Lemma~\ref{lem:inv-comp} the affine subspace $B = \fix(w)$ has
  codimension~$k$.  Next, select a chain of affine subspaces $B=B_k
  \subset B_{k-1} \subset \cdots B_1 \subset B_0 = E$ where the
  subscript indicates its codimension.  By Lemma~\ref{lem:fix-point}
  there is a reflection $r_k$ such that $w_{k-1} = r_k w$ is an
  elliptic with $\fix(w_{k-1}) = B_{k-1}$.  Iteratively applying
  Lemma~\ref{lem:fix-point} we can find reflections $r_{k-1}, \ldots,
  r_2, r_1$ and elements $w_{k-2}, \ldots, w_1, w_0$ where $w_{i-1} =
  r_i w_i$ is an elliptic with $\fix(w_{i-1}) = B_{i-1}$ for $i =
  k,\ldots,2,1$.  In the end, $w_0 = r_1 w_1 = \cdots = r_1 r_2 \cdots
  r_k w$ but $w_0$ is the identity since it fixes all of $E$.
  Rearranging shows $w= r_k \cdots r_2 r_1$.
\end{proof}

\begin{prop}[Hyperbolic upper bound]\label{prop:hyp-upper}
  Every hyperbolic isometry with a $k$-dimensional move-set has a
  length $k+2$ reflection factorization.
\end{prop}

\begin{proof}
  Let $w$ be a hyperbolic isometry whose move-set is $k$-dimensional,
  let $\mov(w) = U +\mu$ be its standard form, and let $w = t_\mu u$
  be the standard splitting of $w$ where $u$ is an elliptic with
  $\mov(u) = U$ (Proposition~\ref{prop:std-split}).  By
  Proposition~\ref{prop:ell-upper} $u$ has a length~$k$ reflection
  factorization and the translation $t_\mu$ is a product of two
  parallel reflections.  Thus $w$ has a length~$k+2$ reflection
  factorization.
\end{proof}

To complete the proof of Theorem~\ref{thm:lr}, we need one more
observation.

\begin{lem}[Parity]\label{lem:parity}
  The lengths of all reflection factorizations of a given element have
  the same parity and the lengths of two elements differing by a
  reflection have opposite parity.  More specifically, if $w$ and $w'$
  are isometries and $r$ is a reflection such that $w' = rw$ then
  $\lr(w') = \lr(w) - 1$ or $\lr(w') = \lr(w) + 1$.
\end{lem}

\begin{proof}
  Partitioning isometries based on whether or not they preserve
  orientation shows that the Cayley graph of $W$ respect to $R$ is a
  bipartite graph and this has the first assertion as a consequence.
  For the second assertion, note that $\lr(w')$ and $\lr(w)$ differ by
  at most $1$ by the way reflection length is defined and parity rules
  out equality.
\end{proof}

\begin{thm}[Reflection length]\label{thm:lr}
  The reflection length of an isometry is determined by its basic
  invariants.  More specifically, let $w$ be an isometry whose
  move-set is $k$-dimensional.  When $w$ is elliptic, $\lr(w) = k$ and
  when $w$ is hyperbolic, $\lr(w) = k+2$.
\end{thm}

\begin{proof}
  For elliptic isometries, Propositions~\ref{prop:lower}
  and~\ref{prop:ell-upper} complete the proof.  For hyperbolic
  isometries, Propositions~\ref{prop:lower} and~\ref{prop:hyp-upper}
  show that $\lr(w)$ is $k+1$ or $k+2$.  The former is ruled out
  because $w$ has a length $k+2$ factorization and by
  Lemma~\ref{lem:parity} $\lr(w)$ has the same parity.
\end{proof}

One corollary of Theorem~\ref{thm:lr} is that the factorizations
produced by Propositions~\ref{prop:ell-upper} and~\ref{prop:hyp-upper}
are now known to be minimal length.  We conclude this section by
characterizing some of the reflections that occur in minimal length
factorizations of a fixed isometry.  For this we need an elementary
observation.

\begin{lem}[Rewriting factorizations]\label{lem:rewriting}
  Let $w = r_1 r_2 \cdots r_k$ be a reflection factorization. For any
  selection $1 \leq i_1 < i_2 < \cdots < i_j \leq k$ of positions
  there is a length~$k$ reflection factorization of $w$ whose first
  $j$ reflections are $r_{i_1} r_{i_2} \cdots r_{i_j}$ and another
  length~$k$ reflection factorization of $w$ where these are the last
  $j$ reflections in the factorization.
\end{lem}

\begin{proof}
  Because reflections are closed under conjugation, for any
  reflections $r$ and $r'$ there exist reflections $r''$ and $r'''$
  such that $r r' = r'' r$ and $r' r = r r'''$.  Iterating these
  rewriting operations allows us to move the selected reflections into
  the desired positions without altering the length of the
  factorization.  
\end{proof}

\begin{defn}[Reflections below $w$]
  Let $r$ be a reflection.  By Lemma~\ref{lem:rewriting}, the
  following conditions are equivalent: (1) $\lr(rw) < \lr(w)$ (2) $r$
  is the leftmost reflection in some minimal length factorization of
  $w$ (3) $r$ is a reflection in some minimal length factorization of
  $w$ (4) $r$ is the rightmost reflection in some minimal length
  factorization of $w$ and (5) $\lr(wr) < \lr(w)$.  When these hold,
  we say that $r$ is a \emph{reflection below $w$}.
\end{defn}

\begin{prop}[Motions and reflections]\label{prop:motion-refl}
  If $w$ is an isometry and $x$ is not fixed by $w$ then the unique
  reflection $r$ that sends $x$ to $w(x)$ is a reflection below $w$.
\end{prop}

\begin{proof}
  That $r$ occurs in some minimal length factorization of $w$ is
  immediate from the flexibility of the constructions used to prove
  Proposition~\ref{prop:ell-upper} and~\ref{prop:hyp-upper}.
\end{proof}

\section{Reflections and Invariants}

In this section we characterize when a reflection $r$ is below an
isometry $w$ in terms of their basic invariants.  We begin with a
corollary of Theorem~\ref{thm:lr}.

\begin{lem}[Move-sets and parity]\label{lem:move-parity}
  For each reflection $r$ and isometry $w$, the dimensions of
  $\mov(w)$ and $\mov(rw)$ have opposite parity.
\end{lem}

\begin{proof}
  By Lemma~\ref{lem:parity} $\lr(w)$ and $\lr(rw)$ have opposite
  parity and by Theorem~\ref{thm:lr} the same holds for the dimensions
  of their move-sets.
\end{proof}

\begin{prop}[Move-sets and reflections]\label{prop:move-refl}
  Let $r$ be a reflection with roots $\pm \alpha$, let $w$ be an
  isometry with $\mov(w) = U+ \mu$, and let $U_\alpha = \spn(U \cup
  \{\alpha\})$.  If $\alpha \in U$ then $\mov(rw)$ is a
  codimension~$1$ subspace of $\mov(w)$.  If $\alpha \not \in U$ then
  $\mov(rw) = U_\alpha +\mu$ and it contains $\mov(w)$ as a
  codimension~$1$ subspace.
\end{prop}

\begin{proof}
  Let $L_\alpha$ be the line spanned by $\alpha$ and let $k$ be the
  dimension of $U$.  Because $r$ only moves points in the $\alpha$
  direction, the move-set of $rw$ is contained in $U_\alpha+\mu$ and
  it must contain at least one point from each of its $L_\alpha$
  cosets.  For $\alpha \not \in U$, this implies that the dimension of
  $\mov(rw)$ is either $k$ or $k+1$.  By Lemma~\ref{lem:move-parity}
  its dimension is $k+1$ and we have $\mov(rw) = U_\alpha +\mu$.  On
  the other hand, for $\alpha \in U$, we have $U_\alpha = U$ and the
  dimension of $\mov(rw)$ is either $k$ or $k-1$.  By
  Lemma~\ref{lem:move-parity} its dimension is $k-1$ and $\mov(rw)$ is
  a codimension~$1$ subspace of $\mov(w)$.
\end{proof}

Proposition~\ref{prop:move-refl} makes it possible to determine how
type and reflection length change when multiplying by a reflection.

\begin{prop}[Hyperbolic isometries and reflections]\label{prop:hyp-refl}
  Let $r$ with reflection with hyperplane $H$ and roots $\pm \alpha$,
  let $w$ be a hyperbolic isometry with $\lr(w)=k$ and $\mov(w) =
  U+\mu$ in standard form, and let $U_\alpha = \spn(U \cup \{\alpha\})$.
  \begin{itemize}
  \item If $\alpha \in U$ then $rw$ is hyperbolic and $\lr(rw) = k-1$.
  \item If $\alpha \not \in U$ and $\mu \in U_\alpha$ then $rw$ is
    elliptic and $\lr(rw) = k-1$.
  \item If $\alpha \not \in U$ and $\mu \not \in U_\alpha$ then $rw$
    is hyperbolic and $\lr(rw) = k+1$.
  \end{itemize}
\end{prop}

\begin{proof}
  When $\alpha$ is in $U$, by Proposition~\ref{prop:move-refl}
  $\mov(rw)$ is a subspace of $\mov(w)$ and since $\mov(w)$ does not
  contain the origin, neither does $\mov(rw)$.  Thus $rw$ is
  hyperbolic.  Similarly, when $\alpha$ is in $U$ the new move-set is
  $U_\alpha +\mu$ which contains the origin iff $\mu \in U_\alpha$ and
  this determines whether $rw$ is elliptic or hyperbolic.  In all
  three cases $\lr(rw)$ is determined by Theorem~\ref{thm:lr}.
\end{proof}

The elliptic analog of Proposition~\ref{prop:hyp-refl} requires more
preparation.

\begin{lem}[Minimal elliptic factorizations]\label{lem:min-ell}
  Let $w =r_1 r_2 \cdots r_k$ be a product of reflections where $r_i$
  has hyperplane $H_i$.  If $w$ is elliptic and $\lr(w)=k$ then the
  roots of these reflections are linearly independent.  Conversely, if
  the roots of these reflections are linearly independent then $w$ is
  elliptic, $\lr(w)=k$, $\fix(w) = H_1 \cap \cdots \cap H_k$ and this
  is one of its minimum length reflection factorizations.
\end{lem}

\begin{proof}
  The factorization shows $\lr(w) \leq k$.  If $w$ is elliptic and
  $\lr(w) = k$ then by Theorem~\ref{thm:lr} its move-set is
  $k$-dimensional and by Proposition~\ref{prop:lower} the roots of the
  reflections are linearly independent.  Conversely, if the roots are
  linearly independent, then their hyperplanes intersect in a
  codimension~$k$ subspace $B = H_1 \cap \cdots \cap H_k$ that is
  fixed by $w$.  Thus $w$ is elliptic.  Moreover, if we start at the
  identity and multply the reflections one at a time in order, then
  the linear independence of the roots and
  Proposition~\ref{prop:move-refl} implies that the move-set of these
  partial products steadily increase.  Thus $\mov(w)$ has dimension
  equal to $k$ and by Lemma~\ref{lem:inv-comp} $\fix(w)$ has
  codimension~$k$.  Since we have already found a subspace fixed by
  $w$ of this dimension, $\fix(w) = B$, $\lr(w) =k$ by
  Theorem~\ref{thm:lr}, and this factorization has minimal length.
\end{proof}

\begin{lem}[Fix-sets and reflections]\label{lem:fix-sets}
  If $w$ and $w'$ are elliptic isometries and $r$ is a reflection such
  that $w' = rw$ then the hyperplane of $r$ intersects both $\fix(w)$
  and $\fix(w')$.
\end{lem}

\begin{proof}
  By Lemma~\ref{lem:parity} $\lr(w') = \lr(w) \pm 1$ and by relabeling
  if necessary we can assume that $\lr(w') = \lr(w) + 1$.  In
  particular, if we set $r_1 = r$ and $w = r_2 \cdots r_k$ is a
  minimal length reflection factorization of $w$ then $w' = rw = r_1
  r_2 \cdots r_k$ is a minimal length reflection factorization of
  $w'$.  By Lemma~\ref{lem:min-ell} $\fix(w') = H_1 \cap \cdots \cap
  H_k$ and $\fix(w) = H_2 \cap \cdots H_k$ where $H_i$ is the
  hyperplane of $r_i$.  This means that $H = H_1$ contains $\fix(w')$
  and, since $\fix(w')$ is nonempty, it intersects $\fix(w)$.
\end{proof}

\begin{prop}[Elliptic isometries and reflections]\label{prop:ell-refl}
  Let $r$ be a reflection with hyperplane $H$ and roots $\pm \alpha$
  and let $w$ be an elliptic isometry with $\lr(w)=k$, $\mov(w)=U$ and
  $\fix(w)=B$.
  \begin{itemize}
  \item If $\alpha \not \in U$ then $rw$ is elliptic and $\lr(rw) = k+1$.
  \item If $\alpha \in U$ and $B \subset H$ then $rw$ is elliptic and
    $\lr(rw) = k-1$.
  \item If $\alpha \in U$ with $B \not \subset H$, then $B$ and $H$
    are disjoint, $rw$ is hyperbolic and $\lr(rw) = k+1$.
  \end{itemize}
\end{prop}

\begin{proof}
  For $\alpha \not \in U$, $rw$ is elliptic because $\mov(rw) = \spn(U
  \cup \{\alpha\})$ by Proposition~\ref{prop:move-refl} and this
  subspace contains the origin.  When $\alpha \in U$ and $B \subset
  H$, $rw$ is elliptic because $B$ is fixed by $rw$.  Finally, suppose
  $\alpha \in U$ and $B \not \subset H$.  By Lemma~\ref{lem:inv-orth},
  $\dir(B) \subset \dir(H)$ so the only way $B$ is not a subset of $H$
  is if it is completely disjoint from $H$.  The isometry $rw$ is
  hyperbolic since by Lemma~\ref{lem:fix-sets} it is not elliptic.  In
  all three cases $\lr(rw)$ is determined by Theorem~\ref{thm:lr}.
\end{proof}

The only situations where the length goes down are the following.

\begin{prop}[Hyperbolic descents]\label{prop:hyp-desc}
  Let $w$ be a hyperbolic isometry and let $r$ be a reflection below
  $w$.  If $rw$ is hyperbolic then $\mov(rw)$ is a codimension~$1$
  subspace of $\mov(w)$, and if $rw$ is elliptic, then
  $\dir(\fix(rw))^\perp = \spn(\mov(w))$.
\end{prop}

\begin{proof}
  When $rw$ is hyperbolic this follows from
  Propositions~\ref{prop:move-refl} and~\ref{prop:hyp-refl}.  When
  $rw$ is elliptic, these same propositions imply that $\mov(rw) =
  \spn(\mov(w))$ and by Lemma~\ref{lem:inv-orth} $\dir(\fix(rw))$ is
  its orthogonal complement.
\end{proof}

\begin{prop}[Elliptic descents]\label{prop:ell-desc}
  If $w$ is a elliptic isometry and $r$ is a reflection below $w$,
  then $rw$ is elliptic and $\fix(rw)$ contains $\fix(w)$ as a
  codimension~$1$ subspace.
\end{prop}

\begin{proof}
  This follows from Propositions~\ref{prop:move-refl}
  and~\ref{prop:ell-refl}.
\end{proof}

\section{Combinatorial Models}

In this section we construct an abstract poset $P$ whose elements are
indexed by affine subspaces of $V$ and $E$.  Its subposets are used to
establish Theorem~\ref{thm:model}, our main result.

\begin{defn}[Global poset]\label{def:global}
  We construct a \emph{global poset} $P$ from two types of elements.
  For each nonlinear affine subspace $M$ in $V$, $P$ contains a
  \emph{hyperbolic} element $h^M$ and for each affine subspace $B$ in
  $E$, $P$ contains an \emph{elliptic} element $e^B$.  We also define
  an \emph{invariant map} $\inv \colon W \onto P$ that sends $w$ to
  $h^{\mov(w)}$ when $w$ is hyperbolic and to $e^{\fix(w)}$ when $w$
  is elliptic.  This explains the names and the notation.  The
  elements of $P$ are ordered as follows.  First, hyperbolic elements
  are ordered by inclusion and elliptic elements by reverse inclusion:
  $h^M \leq h^{M'}$ iff $M \subset M'$ and $e^B \leq e^{B'}$ iff $B
  \supset B'$.  Next, no elliptic element is ever above a hyperbolic
  element.  And finally, $e^B < h^M$ iff $M^\perp \subset \dir(B)$.
  The reader should be careful to note that because $M$ is by
  definition a nonlinear subspace of $V$, the vectors orthogonal to
  all of $M$ are also orthogonal to its span, a linear subspace whose
  dimension is $\dim(M) + 1$.  Transitivity is an easy exercise.
\end{defn}

When $W$ is viewed as a poset under the distance ordering, the
invariant map is an order-preserving map between posets.

\begin{prop}[Order-preserving]\label{prop:op}
  If $w$ is an isometry and $r$ is a reflection with $\lr(rw) <
  \lr(w)$ then $\inv(rw) < \inv(w)$ in $P$.  As a consequence, when
  $W$ is viewed as a poset using the distance ordering, the map $\inv
  \colon W \to P$ is a rank-preserving homomorphism betwen posets.
\end{prop}

\begin{proof}
  This is an immediate consequence of Propositions~\ref{prop:hyp-desc}
  and~\ref{prop:ell-desc} and the observation that the image of a
  covering relation in $W$ is a covering relation in $P$.
\end{proof}

\begin{defn}[Model posets]\label{def:model}
  Let $w \in W$ be an isometry.  By Proposition~\ref{prop:op}, the
  invariant map is order-preserving and thus it sends isometries in
  interval $[1,w]$ to elements less than or equal to $\inv(w)$.  Let
  $P(w)$ denote the subposet of $P$ induced by restricting to
  those elements less than or equal to $\inv(w)$.  We call
  $P(w)$ the \emph{model poset for $w$} and it is a
  \emph{hyperbolic poset} or an \emph{elliptic poset} depending on the
  type of $w$.  Since it is also useful to have a notation for these
  subposets in the absence of an isometry, let $P^M$ denote the
  subposet induced by restricting to those elements less than or equal
  to $h^M$ for a nonlinear affine subspace $M \subset V$ and let $P^B$
  denote the subposet induced by restricting to those elements less
  than or equal to $e^B$ for an affine subspace $B \subset E$.  (This
  notation is unambiguous because $M$ and $B$ are affine subspaces of
  $V$ and $E$, respectively.)  As should be clear, when $w$ is
  hyperbolic $P(w) = P^{\mov(w)}$ and when $w$ is elliptic
  $P(w) = P^{\fix(w)}$.
\end{defn}

The structure of an elliptic poset is straightforward.

\begin{prop}[Elliptic posets]\label{prop:ell-str}
  For each affine subspace $B$ in $E$, the elliptic poset $P^B$ is
  isomorphic to $\lin(U)$ where $U = \dir(B)^\perp$.
\end{prop} 

\begin{proof}
  The poset $P^B$ is essentially the poset of affine subspaces of $E$
  that contain $B$ under reverse inclusion which is isomorphic to
  linear subspaces of $V$ that contain $\dir(B)$ under reverse
  inclusion and thus isomorphic to linear subspaces of $\dir(B)^\perp
  = U$ under inclusion.
\end{proof}

The structure of a hyperbolic poset is only slightly more complicated.

\begin{rem}[Hyperbolic posets]\label{rem:hyp-str}
  Every hyperbolic poset can be decomposed into two subposets whose
  structure is easy to describe.  Let $M$ be a nonlinear affine
  subspace of $V$ and for the moment assume that $M$ has
  codimension~$1$.  From the definition it is clear that the
  hyperbolic elements in the hyperbolic poset $P^M$ form an induced
  subposet isomorphic to $\aff(M)$ and the elliptic elements in $P^M$
  form an induced subposet isomorphic to $\aff(E)^*$ where the
  asterisk indicates that this is the dual of $\aff(E)$ with reverse
  inclusion instead of inclusion.  This is because $\spn(M)$ is all of
  $V$ and $M^\perp = \spn(M)^\perp$ is the trivial subspace.  Thus
  every affine subspace $B$ in $E$ satisfies $M^\perp \subset
  \dir(B)$.  A similar structure is present even when $M$ is not
  codimension~$1$ since the elliptics $e^B$ below $h^M$ are uniquely
  determined by the intersection of $B$ with any fixed affine subspace
  $C$ in $E$ with $\dir(C) = M^\perp$.  Thus, in this case, the
  hyperbolic and elliptic elements of $P^M$ induce subposets that look
  like $\aff(M)$ and $\aff(C)^*$.  From these two basic pieces the
  whole poset is described by declaring that some elements in
  $\aff(M)$ are above specific elements in $\aff(C)^*$.
\end{rem}

\section{Models for Intervals}

In this section we prove that the map $\inv \colon [1,w] \to
P(w)$ is an isomorphism of posets.  This should be slightly
surprising since the invariant map is far from injective in
general. There are many distinct rotations, for example, that rotate
around the same codimension~$2$ subspace.  We begin with three lemmas
about local situations.

\begin{lem}[From elliptic to elliptic]\label{lem:ell-ell}
  If $w$ is an elliptic isometry with $\inv(w) = e^B$ and $C \subset
  E$ is an affine subspace containing $B$ as a codimension~$1$
  subspace, then there is a reflection $r$ such that $u = rw$ and
  $\inv(u) = e^C$.
\end{lem}

\begin{proof} 
  This is a restatement of Lemma~\ref{lem:fix-point} in the new
  terminology.
\end{proof}

\begin{lem}[From hyperbolic to elliptic]\label{lem:hyp-ell}
  If $w$ is a hyperbolic isometry with $\inv(w) = h^M$ and $B$ is an
  affine subspace of $E$ with $M^\perp = \spn(M)^\perp = \dir(B)$,
  then there is a reflection $r$ such that $u = rw$ and $\inv(u) =
  e^B$.
\end{lem}

\begin{proof}
  If $x$ is a point in $B$ then, since $w$ does not fix $x$, the set
  of points equidistant from $x$ and $w(x)$ form a hyperplane $H$.
  Let $r$ be the reflection that fixes $H$, let $\pm \alpha$ be its
  roots and let $u = rw$.  Since some scalar multiple of $\alpha$ lies
  in $M$ and $M$ is nonlinear, $\alpha$ is not in $U = \dir(M)$. By
  Proposition~\ref{prop:move-refl}, $\mov(u) = \spn(M)$ and by
  Lemma~\ref{lem:inv-orth}, $\dir(\fix(u)) = M^\perp = \spn(M)^\perp$.
  Since the affine subspace $B$ is determined by one of its points and
  its space of directions, $\fix(u) = B$ and $\inv(u) = e^B$ as
  required.
\end{proof}

\begin{lem}[From hyperbolic to hyperbolic]\label{lem:hyp-hyp}
  If $w$ is a hyperbolic isometry with $\inv(w) = h^M$ and $M'$ is a
  codimension~$1$ affine subspace of $M = \mov(w)$, then there is a
  reflection $r$ such that $u = rw$ and $\inv(u) = h^{M'}$.
\end{lem}

\begin{proof}
  Let $B$ be the set of all points in $E$ moved by some $\lambda \in
  M'$.  By Proposition~\ref{prop:inv-aff} $B$ is an affine subspace
  and because $M'$ is a proper subspace of $M$, $B$ is a proper
  subspace of $E$.  In fact, because $M'$ has codimension~$1$ in $M$,
  $B$ has codimension~$1$ in $E$, i.e. a hyperplane.  Let $r$ be the
  corresponding reflection and note that its roots $\pm \alpha$ are
  the unit vectors orthogonal to $\dir(M')$ inside $\dir(M)$.  By
  Proposition~\ref{prop:move-refl} $\mov(rw)$ is a codimension~$1$
  subspace of $M$ but since $r$ does not move $B$, $\mov(rw) \supset
  M'$ and thus $\mov(rw) = M'$.
\end{proof}

\begin{prop}[Surjective]\label{prop:inv-surj}
  Let $w \in W$ be an isometry.  For each maximal chain in $P(w)$ from
  $e^E$ to $\inv(w)$ of length~$k$, there is a factorization of $w$ as
  a product of $k$ reflections, $w=r_1r_2 \cdots r_k$, so that the
  suffixes of this factorization are send under the invariant map to
  the elements in this maximal chain.  As a consequence, the map $\inv
  \colon [1,w] \to P(w)$ is surjective.
\end{prop}

\begin{proof}
  In what follows we describe the elements in the given maximal chain
  in descending order so that $\inv(w)$ is its first element and $e^E$
  is its last.  By Lemma~\ref{lem:ell-ell}, Lemma~\ref{lem:hyp-ell} or
  Lemma~\ref{lem:hyp-hyp} depending on the types of the first and
  second elements in the chain, there is a reflection $r_1$ where
  $\inv(r_1w)$ is the second element in the chain.  Applying the same
  lemmas to $r_1w$ means there is a reflection $r_2$ such that
  $\inv(r_2r_1w)$ is the third element in the chain, and so on.  After
  $k$ repetitions, we have found $\{r_i\}$ so that $\inv(r_k \cdots
  r_2 r_1w) = e^E$ but this means that $r_k \cdots r_2 r_1 w$ fixes
  all of $E$, it is the identity and rewriting yields $w = r_1 r_2
  \cdots r_k$.  The intermediate stages are $r_i \cdots r_2 r_1 w =
  r_{i+1} \cdots r_k$ and these are sent by the invariant map to the
  correct elements in the maximal chain by construction.
\end{proof}

\begin{prop}[Injective]\label{prop:inv-inj}
  For each isometry $w \in W$ the map $\inv \colon [1,w] \to P(w)$ is
  injective.
\end{prop}

\begin{proof}
  Let $w$ be an isometry with $\lr(w) =k$ and let $u,u' \in [1,w]$ be
  isometries with $\inv(u) = \inv(u')$.  By definition this means we
  can write $w = uv = u'v'$ with $\lr(w) = \lr(u) + \lr(v) = \lr(u') +
  \lr(v')$.  Because the invariant map is rank-preserving
  (Proposition~\ref{prop:op}), $\lr(u) = \lr(u')$ and the proof is by
  induction on $j=\lr(w) - \lr(u) = \lr(w)-\lr(u') = \lr(v) =
  \lr(v')$.  The base step is trivial since $w = u = u'$ when $j=0$.
  The inductive step splits into two cases.  Case 1: suppose that $u$
  and $u'$ are elliptic with $\inv(u) = \inv(u') = e^B$.  There must
  be a point $x$ in $B$ not fixed by $w$ because $j>0$ implies $w$
  either has a smaller fix-set or it is hyperbolic and fixes no points
  at all.  Since both $u$ and $u'$ fix $x$, both $v$ and $v'$ send $x$
  to $w(x)$.  As a consequence, both $v$ and $v'$ have minimal length
  factorizations that include the unique reflection $r$ sending $x$ to
  $w(x)$ (Proposition~\ref{prop:motion-refl}) and by
  Lemma~\ref{lem:rewriting} we can write $v = v_0r$ and $v' = v'_0r$.
  Thus, $u$ and $u'$ are both below $wr = uv_0 = u'v_0'$ of length
  $\lr(wr) = \lr(w)-1$ and by induction $u=u'$.  Case 2: suppose that
  $u$ and $u'$ are hyperbolic with $\inv(u)=\inv(u') = h^M$ and let $M
  = U + \mu$ be its standard form.  By
  Proposition~\ref{prop:std-split} we can write $w = t_\mu w_0$, $u =
  t_\mu u_0$ and $u' = t_\mu u_0'$ where $w_0$, $u_0$ and $u_0'$ are
  elliptics.  Since $u_0$ and $u_0'$ are below $w_0$, both $\fix(u_0)$
  and $\fix(u_0')$ contain $\fix(w_0)$ and have directions $\dir(u_0)
  = \dir(u_0') = \dir(U)^\perp$.  Since they have points in common and
  the same set of directions, $\fix(u_0) = \fix(u_0')$.  The previous
  case, applied to $w_0$ $u_0$ and $u_0'$ shows that $u_0 = u_0'$ and
  thus $u = u'$.
\end{proof}

\begin{prop}[Inverse]\label{prop:inv-order}
  If $w\in W$ is an isometry and $u, u' \in [1,w]$ are isometries with
  $\inv(u) < \inv(u')$ in $P(w)$, then $u < u'$ in $[1,w]$.
\end{prop}

\begin{proof}
  Extend the chain $e^E \leq \inv(u) < \inv(u') \leq \inv(w)$ to a
  maximal chain in $P(w)$ and apply Proposition~\ref{prop:inv-surj}.
  The result is a factorization $w = r_1r_2 \cdots r_k$ where there is
  a some suffix $w_i = r_{i+1} \cdots r_k$ with $\inv(w_i) = \inv(u')$
  and a shorter suffix $w_j = r_{j+1} \cdots r_k$ sent to $\inv(w_j) =
  \inv(u)$.  By Proposition~\ref{prop:inv-inj}, $w_i = u'$ and $w_j =
  u$.  This means that $u' = r_{i+1} \cdots r_j u$ and $u < u'$ in
  $[1,w]$.
\end{proof}

\begin{thm}[Model posets]\label{thm:model}
  For each isometry $w$, the poset structure of the interval $[1,w]$
  is isomorphic to the model poset $P(w)$.  As a consequence, the
  minimum length reflection factorizations of $w$ are in bijection
  with the maximal chains in $P(w)$.
\end{thm}

\begin{proof}
  By Propositions~\ref{prop:op}, \ref{prop:inv-surj}
  and~\ref{prop:inv-inj} the invariant map is an order-preserving
  bijection from $[1,w]$ to $P(w)$ and by
  Proposition~\ref{prop:inv-order} its inverse is also
  order-preserving.
\end{proof}

\section{Intervals and lattices}

As mentioned in the introduction, a discussion of whether or not the
intervals in $W$ are lattices is included here because these results
are needed elsewhere.  By Theorem~\ref{thm:model} this reduces to the
question of which model posets are lattices. The elliptic case is
straightforward.

\begin{thm}[Elliptic posets are lattices]
  For each affine subspace $B$ in $E$, the elliptic poset $P^B$ is a
  complete lattice.
\end{thm} 

\begin{proof}
  This is an immediate consequence of Proposition~\ref{prop:ell-str}.
\end{proof}

The hyperbolic posets are usually not lattices.  To prove this, we
begin by quoting a definition and a proposition from \cite{BrMc10}.

\begin{defn}[Bowtie]\label{def:bowtie}
  We say that a poset $P$ contains a \emph{bowtie} if there exists a
  $4$-tuple $(a,b:c,d)$ of distinct elements such that $a$ and $b$ are
  minimal upper bounds for $c$ and $d$ and $c$ and $d$ are maximal
  lower bounds for $a$ and $b$.  The name reflects the fact that when
  edges are drawn to show that $a$ and $b$ are above $c$ and $d$, the
  configuration looks like a bowtie.  See
  Figure~\ref{fig:non-lattice}.
\end{defn}

\begin{figure}
  \begin{tikzpicture}
    \coordinate (one) at (0,1.2);
    \coordinate (zero) at (0,-1.2);
    \coordinate (a) at (-.6,.4);
    \coordinate (b) at (.6,.4);
    \coordinate (c) at (-.6,-.4);
    \coordinate (d) at (.6,-.4);
    \draw[-,thick] (b)--(one)--(a)--(c)--(zero)--(d);
    \draw[-,thick] (a)--(d)--(b)--(c);
    \fill (one) circle (.6mm) node[anchor=south] {$1$};
    \fill (a) circle (.6mm) node[anchor=east] {$a$};
    \fill (b) circle (.6mm) node[anchor=west] {$b$};
    \fill (c) circle (.6mm) node[anchor=east] {$c$};
    \fill (d) circle (.6mm) node[anchor=west] {$d$};
    \fill (zero) circle (.6mm) node[anchor=north] {$0$};
  \end{tikzpicture}
  \caption{A bounded graded poset that is not a
    lattice.\label{fig:non-lattice}}
\end{figure}
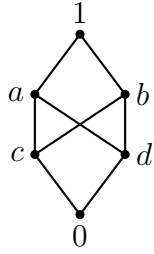

\begin{prop}[Lattice or bowtie]\label{prop:lattice-bowties}
  A bounded graded poset $P$ is a lattice iff $P$ contains no bowties.
\end{prop}

Thus one only needs to determine whether $P^M$ contains any bowties.
Note that maximal lower bounds are easy to calculate.

\begin{rem}[Maximal lower bounds]\label{rem:max-lower}
  Let $M$ be a nonlinear affine subspace of $V$.  Most pairs of
  elements in $P^M$ have a unique maximal lower bound.  For example,
  $e^{B_1} \meet e^{B_2} = e^C$ where $C$ is the affine hull of $B_1
  \cup B_2 \subset E$ and $h^{M_1} \meet e^{B_1} = e^C$ where $C$ is
  the unique affine subspace containing $B_1$ with $\dir(C)$ equal to
  the smallest linear subspace of $V$ containing $\dir(B_1)$ and
  $M_1^\perp$.  In other words, $\dir(C) = \spn(\dir(B_1) \cup
  M_1^\perp)$.  Also, so long as $M_3 = M_1 \cap M_2$ is nonempty,
  $h^{M_1} \meet h^{M_2} = h^{M_3}$.
\end{rem}

The only case not mentioned in Remark~\ref{rem:max-lower} is the
following one.

\begin{prop}[Distinct maximal lower bounds]\label{prop:max-lower}
  If $M$ is a nonlinear affine subspace of $V$ and two elements in
  $P^M$ do not have a unique maximal lower bound, then they are
  hyperbolic elements $h^{M_1}$ and $h^{M_2}$ with $M_1$ and $M_2$
  disjoint and their maximal lower bounds are elliptic elements of the
  form $e^B$ where $\dir(B)^\perp = \dir(M_1) \cap \dir(M_2)$.
\end{prop}

\begin{proof}
  The cases other than two hyperbolics $h^{M_1}$ and $h^{M_2}$ with
  $M_1$ and $M_2$ are disjoint are ruled out by
  Remark~\ref{rem:max-lower}.  The common lower bounds for these two
  are elliptic elements of the form $e^B$ where $\dir(B)^\perp$ is
  contained in $\spn(M_1) \cap \spn(M_2)$ and $e^B$ is maximal when
  $\dir(B)^\perp = \spn(M_1) \cap \spn(M_2)$, but this expression
  simplifies.  Because $M_1$ and $M_2$ are disjoint subsets of $M$,
  $\spn(M_1)$ and $\spn(M_2)$ only intersect inside the linear
  subspace $\dir(M)$.  Also $\spn(M_i) \cap \dir(M) = \dir(M_i)$, so
  we have $\spn(M_1) \cap \spn(M_2) = \dir(M_1) \cap \dir(M_2)$.
\end{proof}

\begin{thm}[Hyperbolic posets are not lattices]\label{thm:hyp-bowtie}
  Let $M$ be a nonlinear affine subspace of $V$.  The poset $P^M$
  contains a bowtie and is not a lattice iff $\dir(M)$ contains a
  proper non-trivial linear subspace $U$, which is true iff the
  dimension of $M$ is at least $2$.  More precisely, for every such
  subspace and for every choice of distinct elements $h^{M_1}$ and
  $h^{M_2}$ with $\dir(M_1) = \dir(M_2) = U$ and distinct elements
  $e^{B_1}$ and $e^{B_2}$ with $\dir(B_1) = \dir(B_2) = U^\perp$,
  these four elements form a bowtie.  Conversely, all bowties in $P^M$
  are of this form.
\end{thm}

\begin{proof}
  Since it is easy to check that the elements listed form a bowtie, we
  focus on establishing the converse.  If $a$ and $c$ are elements in
  $P^M$ with distinct maximal lower bounds then by
  Proposition~\ref{prop:max-lower} they are both hyperbolic with
  disjoint move-sets, say $a = h^{M_1}$ and $b = h^{M_2}$ with $M_1$
  and $M_2$ disjoint.  Moreover, their distinct maximal lower bounds
  are elliptic elements of the form $e^B$ where $\dir(B)^\perp =
  \dir(M_1) \cap \dir(M_2) = U$.  Let $c = e^{B_1}$ and $d= e^{B_2}$
  be two such elements with $B_1$ and $B_2$ distinct and note that
  because $\dir(B_1) = \dir(B_2)$, distinct implies disjoint.  If we
  further assume that $a$ and $b$ are minimal upper bounds for $c$ and
  $d$ then we can conclude that $\dir(M_1) = \dir(M_2) = U$.  The
  distinctness of $a$ and $b$ implies $U$ is a proper subspace of
  $\dir(M)$ and the distinctness of $c$ and $d$ implies $B_i$ is a
  proper subspace of $E$, $\dir(B_i) = U^\perp$ is a proper subspace
  of $V$, and $U$ is nontrivial.
\end{proof}

\section{Lattice completions}\label{sec:complete}

It is well-known that Dedekind used a method of cuts to complete the
rationals to the reals.  Less well-known is that H.~M.~MacNeille was
able to generalize this technique of ``Dedekind cuts'' to show that
every partially ordered set embeds in a complete lattice in an
essentially unique and minimal way.  The resulting complete lattice is
called its Dedekind-MacNeille completion.  We begin by reviewing its
construction as described in \cite{DaPr02} and then apply these
results to the hyperbolic posets that fail to be lattices.

\begin{defn}[Dedekind-MacNeille completion]
  Let $P$ be a poset and for any subset $Q$ in $P$, let $Q^u$ and
  $Q^\ell$ denote the set of upper bounds for $Q$ and lower bounds for
  $Q$, respectively.  The \emph{Dedekind-MacNeille completion} of $P$
  is the collection of subsets $Q$ of $P$ satisfying $Q=(Q^u)^\ell$
  ordered by set inclusion.
\end{defn}

\begin{thm}[Properties of $\DM(P)$]
  For any poset $P$, its Dedekind-MacNeille completion $\DM(P)$ is a
  complete lattice.  Moreover, there is an order-preserving embedding
  $\varphi\colon P\to \DM(P)$ of $P$ into its Dedekind-MacNeille
  completion given by sending each element $p$ of $P$ to $\{p\}^\ell$.
\end{thm}

The Dedekind-MacNeille completion of a poset $P$ can be difficult to
construct from the given definition, particularly when $P$ is
infinite, but there is a characterization theorem which enables one to
recognize $DM(P)$ once it has been constructed by other means.

\begin{defn}[Join-dense and Meet-dense]
  Let $Q$ be a subset of a poset $P$.  We say that $Q$ is
  \emph{join-dense} in $P$ if every element of $P$ is the join of some
  subset of $Q$.  Similarly, $Q$ is \emph{meet-dense} in $P$ if every
  element of $P$ is the meet of some subset of $Q$.
\end{defn}

The following is a restatement of Theorem~2.36 in \cite{DaPr02}.

\begin{thm}[Characterizing $\DM(P)$]\label{thm:dm-char}
  Let $P$ be an ordered set and let $\varphi\colon P\to \DM(P)$ be the
  order-embedding of $P$ into its Dedekind-MacNeille completion
  defined above.  The image of $P$ under $\varphi$ is join-dense and
  meet-dense in the complete lattice $\DM(P)$.  Conversely, if $L$ is
  a complete lattice and $P$ is a subset of $L$ which is both
  join-dense and meet-dense in $L$, then $L$ is order-isomorphic to
  $\DM(P)$ via an order-isomorphism which agrees with $\varphi$ on
  $P$.
\end{thm}

Using this result we are ready to construct the Dedekind-MacNeille
completion of $P^M$ where $M$ is a nonlinear affine subspace of $V$ of
dimension at least $2$.  Given Proposition~\ref{prop:lattice-bowties}
and Theorem~\ref{thm:hyp-bowtie}, it should not be too surprising that
the additional elements are closely related to the locations of the
bowties in $P^M$.

\begin{defn}[Augmenting $P^M$]
  Let $P^M$ be a hyperbolic poset with $M$ a nonlinear affine subspace
  of $V$.  We define a poset $\widetilde P^M$ that contains $P^M$ as
  an induced subposet.  The additional elements are of the form $n^U$
  where $U$ is a proper nontrivial linear subspace of $\dir(M)$ and
  they are ordered by inclusion, i.e. $n^U < n^{U'}$ iff $U \subset
  U'$.  We also set $n^U < h^{M'}$ iff $U \subset \dir(M')$ and $e^B <
  n^U$ iff $\dir(B)^\perp \subset U$.
\end{defn}

\begin{prop}[Complete lattice]\label{prop:complete}
  Every augmented hyperbolic poset is a complete lattice.
\end{prop}

\begin{proof}
  This is straightforward.  As an illustration we sketch the proof
  that arbitrary meets exist and leave the existence of arbitrary
  joins as an easy exercise.  Let $Q \subset \widetilde P^M$ be a
  collection of hyperbolic elements $h^{M_i}$, new elements $n^{U_j}$
  and elliptic elements $e^{B_k}$. If there is at least one elliptic
  in $Q$ then $\bigmeet Q$ is the elliptic $e^C$ where $C$ is the
  smallest affine subspace containing each $B_k$ and where $\dir(C)$
  must contains certain directions determined by the $M_i$ and $U_j$.
  If there are only hyperbolics in $Q$ and the $M_i$'s have a point in
  common then $\bigmeet Q = h^{M'}$ where $M' = \bigcap M_i$.
  Finally, if the $M_i$'s do not have a common point or if there is at
  least one new element in $Q$, then either $\bigmeet Q = n^U$ when $U
  = (\bigcap U_j) \cap (\bigcap \dir(M_i))$ is nontrivial or $\bigmeet
  Q = e^E$ when the subspace $U$ defined in this way is trivial.
\end{proof}

\begin{thm}[Intervals and lattice completions]\label{thm:dm}
  For each affine subspace $M$ in $E$, the augmented poset $\widetilde
  P^M$ is a complete lattice containing $P$ as a meet-dense and
  join-dense subset.  As a consequence, $\widetilde P^M$ is the
  Dedekind-MacNeille completion of $P^M$.
\end{thm}

\begin{proof}
  By Proposition~\ref{prop:complete} and Theorem~\ref{thm:dm-char} we
  only need to show that the new elements are meets and joins of
  elements in $P^M$, but this is trivially true by
  Theorem~\ref{thm:hyp-bowtie} since each new element $n^U$ is
  associated with a bowtie in $P^M$.  In the terminology used there,
  $n^U$ is the meet of $h^{M_1}$ and $h^{M_2}$ and it is the join of
  $e^{B_1}$ and $e^{B_2}$.
\end{proof}

\providecommand{\bysame}{\leavevmode\hbox to3em{\hrulefill}\thinspace}
\providecommand{\MR}{\relax\ifhmode\unskip\space\fi MR }
\providecommand{\MRhref}[2]{%
  \href{http://www.ams.org/mathscinet-getitem?mr=#1}{#2}
}
\providecommand{\href}[2]{#2}

\end{document}